\documentclass{amsart}
\usepackage{amstext,amssymb,amsthm,amsopn,newlfont,graphpap,graphics,graphicx,mathrsfs,enumitem}
\allowdisplaybreaks
\usepackage[parfill]{parskip}
\usepackage[noadjust]{cite}
\usepackage{epigraph}
\usepackage[colorlinks=true,
            linkcolor=red,
            urlcolor=blue,
            citecolor=magenta]{hyperref}
\usepackage{color}
\usepackage{mathrsfs}
\allowdisplaybreaks
\theoremstyle{plain}
\newtheorem{thm}{Theorem}[section]
\newtheorem*{thm*}{Theorem}
\newtheorem{prop}{Proposition}[section]
\newtheorem*{prop*}{Proposition}
\newtheorem{cor}{Corollary}[section]
\newtheorem*{cor*}{Corollary}

\newtheorem*{lem*}{Lemma}
\theoremstyle{definition}
\newtheorem{defn}{Definition}[section]
\newtheorem*{defn*}{Definition}

\newtheorem*{exmp*}{Example}

\newtheorem*{exmps*}{Examples}

\newtheorem*{rem*}{Remark}

\newtheorem*{rems*}{Remarks}

\newtheorem*{note*}{Note}
\newcommand{\N}{{\mathbb N}}
\newcommand{\Z}{{\mathbb Z}}
\newcommand{\R}{{\mathbb R}}
\newcommand{\C}{{\mathbb C}}

\DeclareMathOperator{\Rep}{Re\,}

\newcommand{\eps}{\varepsilon}
\begin{document}
\title[On the Mean Ergodicity of Weak Solutions]
{On the Mean Ergodicity of Weak Solutions\\
of an Abstract Evolution Equation}
\author[Marat V. Markin]{Marat V. Markin}
\address{
Department of Mathematics\newline
California State University, Fresno\newline
5245 N. Backer Avenue, M/S PB 108\newline
Fresno, CA 93740-8001
}
\email{mmarkin@csufresno.edu}
\dedicatory{In loving memory of my teacher, Dr. Miroslav L. Gorbachuk.}
\begin{abstract}
Found are conditions of rather general nature sufficient for the existence of the limit at infinity of the \textit{Ces\`{a}ro means}
\begin{equation*}
\frac{1}{t} \int_0^ty(s)\,ds
\end{equation*}
for every \textit{bounded weak solution} $y(\cdot)$ of the abstract evolution equation
\begin{equation*}
y'(t)=Ay(t),\ t\ge 0,
\end{equation*}
with a closed linear operator $A$ in a Banach space $X$.
\end{abstract}
\maketitle
\epigraph{\textit{Say not in grief he is no more, but live in thankfulness that he was.}}{Hebrew Proverb}

\section[Introduction]{Introduction}

The problem of finding conditions, which secure a certain kind of asymptotic behavior for solutions of evolution equations is pivotal in the qualitative theory of such.

For the abstract evolution equation
\begin{equation}\label{1}
y'(t)=Ay(t),\ t\ge 0,
\end{equation}
with a closed linear operator $A$ in a (real or complex) Banach space $(X,\|\cdot\|)$, we find conditions, formulated exclusively in terms of the operator $A$, the space $X$, or both, which are \textit{sufficient} for the existence of the limit at infinity, in the strong or weak sense, of the \textit{Ces\`{a}ro means}
\begin{equation*}
\frac{1}{t} \int_0^ty(s)\,ds,
\end{equation*}
of the equation's every \textit{bounded weak solution} $y(\cdot)$ ($\displaystyle \sup_{t\ge 0}\|y(t)\|<\infty$) (see Preliminaries).

Observe that the notion of the \textit{Ces\`{a}ro limit} 
\begin{equation*}
\lim_{t\to \infty}\frac{1}{t} \int_0^ty(s)\,ds
\end{equation*}
(weak or strong) extends that of the regular one
\begin{equation*}
\lim_{t\to \infty}y(t),
\end{equation*}
(in the same sense), the existence of the latter implying the existence of the former and its coincidence with the latter. The converse, however, is not true. For instance, in $X=\C$ with the absolute-value norm, all solutions
\[
y(t)=e^{it}f,\ t\ge 0,\ f\in X,
\] 
of equation \eqref{1}, with $A$ being the multiplication operator by the \textit{imaginary unit} $i$, are \textit{bounded} and
\begin{equation*}
\lim_{t\to \infty}\frac{1}{t}\int_0^t e^{is}f\,ds
=\lim_{t\to \infty}\dfrac{e^{it}f-f}{it}=0,
\end{equation*}
whereas
\begin{equation*}
\lim_{t\to \infty}e^{it}f
\end{equation*}
exists only for the trivial one ($f=0$).

The results obtained in \cite{Markin1991} for \textit{classical solutions} of \eqref{1} deal with the cases of
\begin{itemize}
\item a \textit{finite-dimensional} space $(X,\|\cdot\|)$,
\item a \textit{continuously} or {\it reducibly invertible} operator $A$, and
\item a \textit{reflexive} space $(X,\|\cdot\|)$,
with $A$ being invertible or generating a certain direct sum decomposition for $X$,
\end{itemize}
and are generalized in \cite{Markin1994Preprint,Markin1994Disser}
to \textit{weak solutions} with the added case of 
\begin{itemize}
\item a \textit{normal operator} $A$ in a complex Hilbert space
$(X,(\cdot,\cdot),\|\cdot\|)$.
\end{itemize}

The purpose of the present paper is to publish for the first time largely revised results on the mean ergodicity of \textit{weak solutions}, which have only seen a very limited printing in the form of the preprint \cite{Markin1994Preprint} and the abstract to the dissertation \cite{Markin1994Disser} (without proof) so far, along with some fresh ones, including the case of a {\it scalar type spectral operator} $A$ in a complex Banach space $(X,\|\cdot\|)$ generalizing and replacing that of a \textit{normal operator}. The reference base has been thoroughly upgraded considering later developments such as \cite{Markin2002(1),Markin2002(2),Markin2004(1),Markin2004(2),Markin2006,Markin2017(1)}. 

\section[Preliminaries]{Preliminaries}

Henceforth, $A$ is supposed to be a \textit{closed linear operator} in a (real or complex) Banach space $(X,\|\cdot\|)$.

\begin{defn}[Weak Solution]\label{def1}\ \\
A strongly continuous vector function $y:[0,\infty)\mapsto X$ is called a \textit{weak solution} of equation \eqref{1} if, for all $t\ge 0$,
\begin{equation}\label{int}
\int_0^ty(s)\,ds\in D(A)\ \text{and} \ y(t)=y(0)+A\int_0^ty(s)\,ds,
\end{equation}
where $D(\cdot)$ is the domain of an operator.
\end{defn}

The operator $A$ being \textit{densely defined}, by {\cite[Lemma]{Ball}} (cf. {\cite[Lemma VI.1.4]{Goldberg}}), it can be easily shown that the \textit{closedness} of $A$ affords the following equivalent definition according to \cite{Ball}.

\begin{defn}[Weak Solution]\label{def2}\ \\
A strongly continuous vector function $y:[0,\infty)\rightarrow X$ is called a {\it weak solution} of equation \eqref{1} if, for any $g^* \in D(A^*)$,
\begin{equation*}
\dfrac{d}{dt}\langle y(t),g^*\rangle = \langle y(t),A^*g^* \rangle,\ t\ge 0,
\end{equation*}
where $A^*$ is the operator {\it adjoint} to $A$ and $\langle\cdot,\cdot\rangle$ is the {\it pairing} between $X$ and its dual space $X^*$.
\end{defn}

The solutions of equation \eqref{1} in the sense of Definition \ref{def1}, in which the existence of the adjoint operator $A^*$ is not required, and hence, $A$ need not be densely defined, are also called \textit{``mild solutions"} (cf. {\cite[Definition II.6.3]{Engel-Nagel}}). Here, for consistency, we adhere to the term \textit{``weak solutions"} as in \cite{Markin1994Preprint,Markin1994Disser,Markin2002(1)}.

Observe that, \textit{a priori}, the {\it weak solutions} of equation \eqref{1} need not be differentiable in the strong sense or take values in the domain $D(A)$ of the operator $A$.

The notion of the \textit{weak solution} generalizes that of the \textit{classical} one, strongly differentiable on $[0,\infty)$ and satisfying the equation in the traditional plug-in sense, the {\it classical solutions} being precisely the \textit{weak} ones strongly differentiable on $[0,\infty)$.
For instance, if the operator $A$ generates a $C_0$-semigroup $\left\{T(t)\right\}_{t\ge 0}$, which is equivalent to the \textit{well-posedness} of the associated \textit{abstract Cauchy problem}
\begin{equation}\label{ACP}
\begin{cases}
y'(t)=Ay(t),\ t\ge 0,\\
y(0)=f
\end{cases}
\end{equation}
in the sense of 
{\cite[Definition II.6.8]{Engel-Nagel}}, the {\it general weak solution} of equation \eqref{1} is of the form
\begin{equation*}
y(t)=T(t)f,\ t\ge 0,f\in X,
\end{equation*}
{\cite[Proposition II.6.4]{Engel-Nagel}} (see also {\cite[Theorem]{Ball}}), whereas the {\it general classical solution} of \eqref{1} is of the form
\begin{equation*}
y(t)=T(t)f,\ t\ge 0,f\in D(A),
\end{equation*}
{\cite[Proposition II.6.3]{Engel-Nagel}}
(see also \cite{Hille-Phillips}), the two, by the \textit{Closed Graph Theorem}, being the same \textit{iff} $A$ is a \textit{bounded} operator on $X$, in which case
\begin{equation*}
T(t)=e^{tA}:=\sum_{n=0}^\infty \dfrac{t^n}{n!}A^n,\ t\ge 0.
\end{equation*}

As is known \cite{Goldstein-Radin-Showalter,Lin-Shaw,Shaw}, if the operator $A$ generates a {\it bounded} $C_0$-semigroup $\left\{T(t)\right\}_{t\ge 0}$, which immediately implies boundedness for all weak solutions of equation \eqref{1}, 
\begin{equation*}
s\mbox{-}\!\!\lim_{t \to \infty}\frac{1}{t} \int_0^t T(s)f\,ds,
\end{equation*}
($\displaystyle s\mbox{-}\!\!\lim_{t \to \infty}$ stands for the \textit{strong limit}) exists
for each $f\in X$ \textit{iff} $X$ is decomposable into the direct sum
\begin{equation}\label{decomp1}
X=\ker A \oplus \overline{R(A)},
\end{equation}
where $\ker A$ is the {\it kernel} of $A$ and $\overline{R(A)}$ is the closure of its {\it range}, $R(A)$. In this case,
\begin{equation*}
s\mbox{-}\!\!\lim_{t \to \infty}\frac{1}{t} \int_0^t T(s)f\,ds=Pf,
\end{equation*}
where $P$ is the \textit{projection operator} onto $\ker A$ along $\overline{R(A)}$, called the \textit{mean ergodic projection} of the semigroup. The space $X$ being {\it reflexive}, decomposition \eqref{decomp1} holds automatically
(cf. \cite{Hille-Phillips,Komatsu1969,Butyrin1994}).

Recall that, we impose conditions on the space $X$ and the operator $A$ only. The latter example does fall into this framework, the operator's $A$ generating a {\it bounded} $C_0$-semigroup being characterized by the corresponding case of the \textit{Generation Theorem} {\cite[Theorem II.3.8]{Engel-Nagel}} in terms of the location of its \textit{spectrum}, $\sigma(A)$, and certain growth estimates for all natural powers of its \textit{resolvent}, 
\[
R(\lambda,A):=(A-\lambda I)^{-1}
\]
($I$ is the \textit{identity operator} on $X$). 

Observe also that, in our discourse, Cauchy problem \eqref{ACP} associated with abstract evolution
equation \eqref{1} may be well- or ill-posed. In the latter case, the only bounded weak solutions, whose existence is guaranteed, are the \textit{equilibrium solutions}
\begin{equation*}
y(t)= f\in \ker A,\ t\ge 0,
\end{equation*}
and the \textit{eigenvalue solutions}
\begin{equation*}
y(t)=e^{\lambda t} f,\ t\ge 0,
\end{equation*}
corresponding to the nonzero eigenvalues
$\lambda$ of $A$ with $\Rep\lambda\le 0$, if any.  

In what follows, the notations 
$\displaystyle s\mbox{-}\!\!\lim_{t \to \infty}$ 
and 
$\displaystyle w\mbox{-}\!\!\lim_{t \to \infty}$
are used for the limits in the \textit{strong} and \textit{weak} sense, $\rho(\cdot)$ and $\sigma(\cdot)$ designate the \textit{resolvent set} and \textit{spectrum} of an operator, respectively,
and $L(X)$ denotes the space of bounded linear operators on $(X,\|\cdot\|)$. 

\section{Mean Ergodicity of a Particular Weak Solution}

In the aforementioned case of $A$ generating a bounded $C_0$-semigroup in a complex Banach space, the \textit{Ces\`{a}ro means} of every bounded weak solution of equation \eqref{1} converge at infinity to an \textit{equilibrium state}. The following statement confirms that this profound fact is not coincidental.

\begin{prop}[Mean Ergodicity of a Particular Weak Solution]\label{part}\ \\
Let $y(\cdot)$ be a bounded weak solution of equation \eqref{1} with a closed linear operator $A$ in a (real or complex) Banach space $(X,\|\cdot\|)$. If, for a sequence $\left\{t_n\right\}_{n=1}^\infty\subset (0,\infty)$ with $t_n\to \infty$, $n\to\infty$,
\begin{equation}\label{weak1}
w\mbox{-}\!\!\lim_{n \to \infty}\frac{1}{t_n} \int_0^{t_n}y(s)\,ds=y_\infty\in X,
\end{equation}
then $y_\infty\in\ker A$. 
\end{prop}

\begin{proof}\quad
By \eqref{int}, for all $n\ge 1$,
\begin{equation*}
\int_0^{t_n}y(s)\,ds\in D(A)\quad \text{and} \quad \dfrac{1}{t_n}[y(t_n)-y(0)]=A\left[\dfrac{1}{t_n}\int_0^{t_n}y(s)\,ds\right].
\end{equation*}

Whence, considering the \textit{boundedness} of $y(\cdot)$, we infer that
\begin{equation}\label{Astrong1}
s\mbox{-}\!\!\lim_{n \to \infty}A\left[\dfrac{1}{t_n}\int_0^{t_n}y(s)\,ds\right]=0.
\end{equation}

Since, as follows from the \textit{Hahn-Banach Theorem}, the graph of the closed operator $A$ is also \textit{weakly closed} (see, e.g., \cite{Dun-SchI}), \eqref{weak1} and \eqref{Astrong1} jointly imply that
\begin{equation*}
y_\infty\in D(A)\ \text{and}\ Ay_\infty=0,
\end{equation*}
and hence, $y_\infty\in \ker A$.
\end{proof}

We instantly obtain the following

\begin{cor}
Let $y(\cdot)$ be a bounded weak solution of equation \eqref{1} with a closed linear operator $A$ in a (real or complex) Banach space $(X,\|\cdot\|)$. If, for a sequence $\left\{t_n\right\}_{n=1}^\infty\subset (0,\infty)$ with $t_n\to \infty$, $n\to\infty$,
\begin{equation*}
w\mbox{-}\!\!\lim_{n\to \infty}y(t_n)=y_\infty\in X,
\end{equation*}
then $y_\infty\in\ker A$. 
\end{cor}

\section{Finite-Dimensional Space}

\subsection{Preliminaries}

Let $(X,\|\cdot\|)$ be a complex finite-dimensional Banach space with $\dim X=n$ ($n\in \N$) and $A\in L(X)$. 

Observe that, the strong and weak topologies on $X$ coinciding (see, e.g., \cite{Dun-SchI}), all {\it weak solutions} of equation \eqref{1} are {\it classical} ones, strong and weak limits are indistinguishable and we can use the notation $\displaystyle \lim_{t \to \infty}$ to stand for either one. 

Furthermore, the operator $A$ admits the following \textit{spectral decomposition} readily obtained from its Jordan canonical matrix representation (see, e.g., \cite{Horn-Johnson}):
\begin{equation}\label{sd}
A=\sum_{j=1}^m[\lambda_jP_j+Q_j],
\end{equation}
where 
\begin{itemize}
\item $\lambda_j$, $j=1,\dots,m$ ($m\in \N$, $1\le m\le n$), are distinct {\it eigenvalues} of $A$ forming its \textit{spectrum}, $\sigma(A)$,
\item $P_j$, $j=1,\dots,m$, are {\it projection operators}, and 
\item $Q_j:=(A-\lambda_jI)P_j=P_j(A-\lambda_jI)$, $j=1,\dots,m$, are {\it nilpotent operators}
\end{itemize}
(see, e.g., \cite{Dun-SchI,Kato,Glazman-Lyubich,Daletsky-Krein}).

A few important observations concerning the structure of the \textit{spectral decomposition} are in order.

\begin{itemize}
\item The projection $P_j$, $j=1,\dots,m$, is the \textit{spectral projection}, or {\it Riesz projection}, of $A$ at $\lambda_j$. Recall that the latter is defined for an arbitrary $\lambda\in \C$ in the sense of the Dunford-Riesz operational calculus as
\begin{equation*}
P(\lambda,A)=-\dfrac{1}{2\pi i}\int\limits_{\gamma} R(z,A)\,dz,
\end{equation*}
where $\gamma$ is a positively oriented rectifiable Jordan contour enclosing $\lambda$, which along with its interior, except, possibly, for $\lambda$, is contained in the \textit{resolvent set} $\rho(A)$ of $A$, and $R(\cdot,A)$ is the \textit{resolvent function} of $A$ \cite{Dun-SchI}. 

Observe that, $P(\lambda,A)=0$ \textit{iff} $\lambda\in \rho(A)$. Provided $\lambda$ is an {\it eigenvalue} of $A$, the range $R(P(\lambda,A))$ of $P(\lambda,A)$ is a subspace in $X$, 
which is not to be confused with the {\it eigenspace} of $\lambda$, $\ker(A-\lambda I)$. In fact, the former contains the latter, $\dim R(P(\lambda,A))$ being the {\it algebraic multiplicity} of $\lambda$, i.e., its multiplicity as a zero of the {\it characteristic polynomial} 
of $A$, or the sum of the sizes of all Jordan blocks corresponding to $\lambda$, and $\dim\ker(A-\lambda I)$ being the {\it geometric multiplicity} of $\lambda$, or 
the number of Jordan blocks corresponding to $\lambda$ (see, e.g., \cite{Horn-Johnson}). In fact,
\[
\dim R(P(\lambda,A))=\dim \ker(A-\lambda I)
\]
\textit{iff} all Jordan blocks of $\lambda$ are of size 1 (see, e.g., \cite{Dun-SchI,Horn-Johnson}).
\item The operator $Q_j$, $j=1,\dots,m$, is {\it nilpotent}, the \textit{index} of $\lambda_j$,
\begin{equation}\label{nilp}
k_j:=\min \left\{k\in \N\;\middle|\;Q_j^k=0\right\}
\end{equation} 
being the multiplicity of $\lambda_j$ as a zero of the \textit{minimal polynomial} of $A$, which is the size of the largest Jordan block of $\lambda_j$ \cite{Glazman-Lyubich,Horn-Johnson,Kato}.

Observe that $R(P(\lambda_j,A))=\ker(A-\lambda_j I)^{k_j}$, $j=1,\dots,m$.
\item The operators $P_j$ and $Q_j$,
$j=1,\dots,m$, are bound by the following relations:
\begin{equation}\label{relations}
\begin{split}
&\sum_{j=1}^m P_j=I,\\
&P_iP_j=\delta_{ij}P_i,\ i,j=1,\dots,m,\\
&P_iQ_j=Q_jP_i=\delta_{ij}Q_j,\ i,j=1,\dots,m,\\
&Q_iQ_j=\delta_{ij}Q_i^2,\ i,j=1,\dots,m,
\end{split}
\end{equation} 
($\delta_{ij}$ is the {\it Kronecker delta}).
\end{itemize}

The solutions of equation \eqref{1} are given by the familiar exponential formula
\begin{equation}\label{exp0}
y(t)=e^{tA}f,\ t\ge 0,f\in X,
\end{equation} 
where
\begin{equation*}
e^{tA}=\sum_{k=0}^\infty\dfrac{t^k}{k!}A^k,\ t\ge 0,f\in X,
\end{equation*} 
which, due to \eqref{sd}--\eqref{relations}, 
can be rewritten as the finite sum
\begin{equation}\label{exp}
e^{tA}=\sum_{j=1}^m e^{t\lambda_j}P_j\sum_{k=0}^{k_j-1} \dfrac{t^k}{k!}Q_j^k=\sum_{j=1}^m \sum_{k=0}^{k_j-1}e^{t\lambda_j} \dfrac{t^k}{k!}(A-\lambda_jI)^kP_j,\ t\ge 0,
\end{equation}
(cf. {\cite[Proposition I.2.6]{Engel-Nagel2006}}, see
also \cite{Glazman-Lyubich,Daletsky-Krein,Kato}).

This representation, instrumental in proving the classical \textit{Lyapunov stability theorem} \cite{Lyapunov1892} 
(cf. {\cite[Theorem I.2.10]{Engel-Nagel}}), is used to prove the succeeding statement.

\subsection{Main Statement}

\begin{thm}\label{Thm1}
Let $(X,\|\cdot\|)$ be a complex finite-dimensional Banach space and $A\in L(X)$. Then, for each bounded (weak) solution $y(\cdot)$ of equation \eqref{1},   
\begin{equation*}
\lim_{t\to\infty}\dfrac{1}{t}\int_0^ty(s)\,ds=P(0,A)y(0),
\end{equation*} 
where $P(0,A)$ is the spectral projection of $A$ at $0$.

The spectrum of the operator $A$ containing no pure imaginary values,
\begin{equation*}
\lim_{t\to\infty}y(t)=P(0,A)y(0).
\end{equation*} 
\end{thm}

\begin{proof}\quad 
Let $y(\cdot)$ be an arbitrary bounded weak solution of equation \eqref{1}. Then, by \eqref{exp0} and \eqref{exp},
\begin{equation}\label{exp1}
y(t)=e^{tA}f=\sum_{j=1}^m e^{t\lambda_j}P_j\sum_{k=0}^{k_j-1} \dfrac{t^k}{k!}Q_j^k f,\ t\ge 0,
\end{equation}
with $y(0)=f\in X$.

Since, by \eqref{relations}, the spectral projections $P_j$, $j=1,\dots,m$, form a {\it resolution of the identity}, we can introduce a new norm on $X$ as follows:
\begin{equation*}
X\ni g\mapsto\|g\|_1:=\sum_{j=1}^m\|P_jg\|,
\end{equation*}
which, considering that $X$ is \textit{finite-dimensional}, is \textit{equivalent} to the original one (see, e.g., \cite{Trenogin}).

This implies that the boundedness of $y(\cdot)$ is equivalent to the boundedness of each summand 
\begin{equation}\label{summand}
e^{t\lambda_j}P_j\sum_{k=0}^{k_j-1} \dfrac{t^k}{k!}Q_j^k f,\ t\ge 0,\
j=1,\dots,m,
\end{equation}
in representation \eqref{exp1}.

For each $j=1,\dots,m$, we have the following cases:
\begin{itemize}
\item[(1)] If $\Rep\lambda_j<0$, 
the corresponding summand \eqref{summand} is, obviously, boun\-ded and converges to $0$ as $t\to\infty$.
\item[(2)] If $\Rep\lambda_j\ge 0$, since, in view of \eqref{relations},
\[
\left\|e^{t\lambda_j}P_j\sum_{k=0}^{k_j-1} \dfrac{t^k}{k!}Q_j^k f\right\|
\ge e^{t\Rep\lambda_j}\left[\dfrac{t^{k_j-1}}{(k_j-1)!}\left\|Q_j^{k_j-1}P_jf\right\|
-\sum_{k=0}^{k_j-2}\dfrac{t^k}{k!}\left\|Q_j^kP_jf\right\|\right],\ t\ge 0,
\]
and hence, the boundedness of the summand necessarily implies that 
\begin{equation*}
Q_j^{k_j-1}P_jf=0.
\end{equation*}

Continuing in this fashion, we arrive at the following conclusion:
\begin{itemize}
\item if $\Rep\lambda_j>0$,
\begin{equation*}
Q_j^kP_jf=0,\ k=0,\dots,k_j-1,
\end{equation*}
i.e., $P_jf=0$, and hence,
\begin{equation*}
e^{t\lambda_j}P_j\sum_{k=0}^{k_j-1} \dfrac{t^k}{k!}Q_j^k f=0,\ t\ge 0;
\end{equation*}
\item if $\Rep\lambda_j=0$, 
\begin{equation*}
Q_j^kP_jf=0,\ k=1,\dots,k_j-1,
\end{equation*}
i.e.,
\begin{equation*}
e^{t\lambda_j}P_j\sum_{k=0}^{k_j-1} \dfrac{t^k}{k!}Q_j^k f=e^{t\lambda_j}P_jf,\ t\ge 0.
\end{equation*}
\end{itemize}
\end{itemize}

Thus, for a bounded weak solution $y(\cdot)$ of \eqref{1}, representation \eqref{exp1} acquires the form
\begin{equation}\label{expb}
y(t)=\sum_{j:\,\Rep\lambda_j<0}e^{t\lambda_j}P_j\sum_{k=0}^{k_j-1} \dfrac{t^k}{k!}Q_j^kf
+
\sum_{j:\,\Rep\lambda_j=0}e^{t\lambda_j}P_jf,\ t\ge 0,
\end{equation}
in which 
\begin{itemize}
\item the sum corresponding to the eigenvalues of $A$ with negative real part, obviously, vanishes at infinity:
\begin{equation*}
\lim_{t\to\infty}\sum_{j:\,\Rep\lambda_j<0}e^{t\lambda_j}P_j\sum_{k=0}^{k_j-1} \dfrac{t^k}{k!}Q_j^kf=0;
\end{equation*}
\item the sum corresponding to the \textit{pure imaginary} eigenvalues of $A$ vanishes at infinity in the Ces\'aro sense since, for each $\lambda_j\in i\R\setminus \left\{0\right\}$,
\begin{equation*}
\lim_{t\to\infty}\dfrac{1}{t}\int_0^t e^{s\lambda_j}P_jf\,ds
=\lim_{t\to\infty}\dfrac{e^{t\lambda_j}-1}{t\lambda_j}P_jf=0;
\end{equation*}
\item provided, for some $j=1,\dots,m$, $\lambda_j=0$, the corresponding constant term in \eqref{expb} is $P_jf$, where $P_j=P(0,A)$.

If $0\in \rho(A)$, as we noted above, $P(0,A)=0$.
\end{itemize}

Whence, the conclusion of the statement follows immediately.
\end{proof}

\section{Continuously/Reducibly Invertible Operator}

\subsection{Preliminaries}

Now, without imposing any restrictions on the space $(X,\|\cdot\|)$, we require that the closed operator $A$ be \textit{reducibly invertible}, i.e.,
that its range $R(A)$ be a \textit{closed subspace} in $X$ and the direct sum decomposition
\begin{equation}\label{ri}
X=\ker A\oplus R(A)
\end{equation}
hold (see, e.g., \cite{Korolyuk-Turbin}, also \cite{Kato,Daletsky-Krein}).

The reducible invertibility of $A$ gives rise to the fact that
\begin{equation*}
(A+P)^{-1}\in L(X).
\end{equation*}

We show it here for the reader's convenience.

Let $P$ be the projection operator onto $\ker A$ along $R(A)$ and suppose that, for some $f\in D(A)$, 
\[(A+P)f=0.\]
Then $Af=-Pf$, which, since the subspaces $\ker A$ and $R(A)$ are \textit{disjoint}, implies 
\[Af=Pf=0.\]
Whence, we infer that $f\in \ker A$, and therefore, $f=Pf=0$. 

Thus, the closed linear operator $A+P$ has an inverse $(A+P)^{-1}$.

Now, let us show that 
\[R(A+P)=X.\]
Indeed, for an arbitrary $f\in X$, in view of \eqref{ri},
\[f=f_1+Af_2\]
with $f_1\in \ker A$ and $f_2\in D(A)$. Further,
\[f_2=f_3+f_4,\]
where $f_3\in \ker A$ and $f_4\in R(A)$. Since $f_4=f_2-f_3\in D(A)\cap R(A)$ and $Af_4=Af_2$, we have:
\[(A+P)(f_1+f_4)=Pf_1+Af_4=f_1+Af_2=f.\]

Hence, the inverse $(A+P)^{-1}$, which is a closed linear operator defined on the whole space $X$, is \textit{bounded} by the {\it Closed Graph Theorem}.

\subsection{Main Statement}

\begin{thm}\label{Thm2}
Let $A$ be a closed linear operator in a (real or complex) Banach space $(X,\|\cdot\|)$.
\begin{enumerate}[label=(\roman*)]
\item If the operator $A$ has a bounded inverse $A^{-1}$, for each bounded weak solution $y(\cdot)$ of equation \eqref{1},
\[
s\mbox{-}\!\!\lim_{t\to\infty}\dfrac{1}{t}\int_0^ty(s)\,ds=0
\quad \text{with}\quad \left\|\dfrac{1}{t}\int_0^ty(s)\,ds\right\|
=\mathcal{O}\left(\dfrac{1}{t}\right),\ t\to\infty.
\]
\item If the operator $A$ is reducibly invertible, for each bounded weak solution $y(\cdot)$ of equation \eqref{1},
\[
s\mbox{-}\!\!\lim_{t\to\infty}\dfrac{1}{t}\int_0^ty(s)\,ds=Py(0),
\]
where $P$ is the projection onto $\ker A$ along $R(A)$, with 
\[
\left\|\dfrac{1}{t}\int_0^ty(s)\,ds-Py(0)\right\|
=\mathcal{O}\left(\dfrac{1}{t}\right),\ t\to\infty.
\]
\end{enumerate}
\end{thm}

\begin{proof}\
\begin{enumerate}[label=(\roman*)]
\item Observe that the bounded inverse $A^{-1}$ need not be defined on the whole $X$ and let $y(\cdot)$ be an arbitrary bounded weak solution of equation \eqref{1}.

Then, for all $t\ge 0$, \eqref{int} holds, and hence,
\begin{equation*}
\dfrac{1}{t}\int_0^ty(s)\,ds=\dfrac{1}{t}A^{-1}[y(t)-y(0)],\ t>0,
\end{equation*}
which implies 
\begin{equation}\label{est2}
\biggl\|\dfrac{1}{t}\int_0^ty(s)\,ds\biggr\|\le\|A^{-1}\|\|y(t)-y(0)\|\dfrac{1}{t}
\le 2\|A^{-1}\|\sup_{s\ge 0}\|y(s)\|\dfrac{1}{t},\ t>0,
\end{equation}
where
\[
\|A^{-1}\|:=\sup_{g\in R(A),\ \|g\|=1}\|A^{-1}g\|.
\]

Whence,
\begin{equation*}
\left\|\dfrac{1}{t}\int_0^ty(s)\,ds\right\|
=\mathcal{O}\left(\dfrac{1}{t}\right),\ t\to\infty.
\end{equation*}
\item Let $y(\cdot)$ be an arbitrary bounded weak solution of equation \eqref{1}.

The vector function
\[u(t):=y(t)-Py(0),\ t\ge 0,\]
obtained by combining the bounded weak solution
$y(\cdot)$ of equation \eqref{1} with its \textit{equilibrium solution} $Py(0)$, as readily follows from the linear homogeneity of equation \eqref{1}, reflected in Definition \ref{def1}, is also a bounded weak solution of \eqref{1}, i.e.,
for all $t\ge 0$,
\begin{equation}\label{uws}
\int_0^t u(s)\,ds\in D(A)\ \text{and} \ u(t)=u(0)+A\int_0^t u(s)\,ds.
\end{equation}

Since, by decomposition \eqref{ri}, 
\[
u(0)=y(0)-Py(0)\in R(A),
\]
in view of \eqref{uws}, we infer that
\[u(t)\in R(A),\ t\ge 0,\] 
which, by the \textit{closedness} of $R(A)$, implies that
\begin{equation*}
\int_0^t u(s)\,ds\in D(A)\cap R(A),\ t\ge 0,
\end{equation*}
and hence,
\begin{equation}\label{null1}
P\int_0^t u(s)\,ds=0,\ t\ge 0.
\end{equation}

Now, let us show that $u(\cdot)$ is also a bounded weak solution of the evolution equation
\begin{equation}\label{2}
y'(t)=(A+P)y(t),\ t\ge 0.
\end{equation}

Indeed, by \eqref{uws},
\begin{equation*}
\int_0^t u(s)\,ds\in D(A)=D(A+P),\ t\ge 0,
\end{equation*}
and, considering \eqref{null1} 
\begin{equation*}
u(t)=u(0)+A\int_0^t u(s)\,ds=u(0)+(A+P)\int_0^t u(s)\,ds,\ t\ge 0.
\end{equation*}

The operator $A+P$ having a bounded inverse $(A+P)^{-1}\in L(X)$, according to proved part (i),
\begin{equation*}
\left\|\dfrac{1}{t}\int_0^ty(s)\,ds-Py(0)\right\|
=\left\|\dfrac{1}{t}\int_0^t u(s)\,ds\right\|
=\mathcal{O}\left(\dfrac{1}{t}\right),\ t\to\infty.
\end{equation*}
\end{enumerate}
\end{proof}

\subsection{Concluding Remarks}\label{cr}
\begin{enumerate}[label=\arabic*.]
\item Theorem \ref{Thm2}, proving the strong convergence at infinity
of the \textit{Ces\`{a}ro means}
\[
\dfrac{1}{t}\int_0^t y(s)\,ds
\]
for every bounded weak solution $y(\cdot)$ of equation \eqref{1}, also provides information about the speed of this convergence. 
\item In all, no part of Theorem \ref{Thm2} is more general than the other.

Indeed, in $X=l_2$ (the space of square-summable sequences), 
\begin{itemize}
\item the right-shift operator 
\begin{equation*}
A\{x_1,x_2,x_3,\dots\}:=\{0,x_1,x_2,\dots\},
\ \{x_1,x_2,x_3,\dots\}\in l_2,
\end{equation*}
has the bounded inverse
\begin{equation*}
A^{-1}\{0,y_2,y_3,\dots\}:=\{y_2,y_3,y_4,\dots\},
\ \{0,y_2,y_3,\dots\}\in R(A),
\end{equation*}
but, $R(A)$ being a proper closed subspace in $l_2$ and $\ker A=\{0\}$, decomposition \eqref{ri} does not hold;
\item the bounded linear operator
\begin{equation*}
A\{x_1,x_2,x_3,\dots\}:=\{x_1,0,x_3,0,\dots\},
\ \{x_1,x_2,x_3,\dots\}\in l_2,
\end{equation*}
is {\it reducibly invertible}, but has no inverse.
\end{itemize}
\item When $(X,\|\cdot\|)$ is a finite-dimensional space, part (ii) of Theorem \ref{Thm2} is, obviously, more general than part (i) and, the space being complex, is consistent with Theorem \ref{Thm1}, which follows from the fact that, in such a space, decomposition \eqref{ri} holds \textit{iff} $0$ is either a \textit{regular point} of $A$ or an \textit{eigenvalue}, whose {\it index} is equal to $1$, i.e., all Jordan blocks corresponding to $0$
are of size 1, and, in both cases, $P=P(0,A)$ \cite{Dun-SchI,Glazman-Lyubich}.
\end{enumerate}

\section{Reflexive Space}

\subsection{Main Statement}

Here, we assume the space $(X,\|\cdot\|)$ to be {\it reflexive}, resting our argument upon the characteristic property of such spaces that each bounded sequence of elements contains a weakly convergent subsequence (see, e.g., \cite{Dun-SchI}).

\begin{thm}\label{Thm3} 
Let $A$ be a closed linear operator in a (real or complex) reflexive Banach space $(X,\|\cdot\|)$.

\begin{enumerate}[label=(\roman*)]
\item If the operator $A$ has an inverse $A^{-1}$, 
for each bounded weak solution $y(\cdot)$ of equation \eqref{1},
\begin{equation*}
w\mbox{-}\!\!\lim_{t\to\infty}\dfrac{1}{t}\int_0^ty(s)\,ds=0.
\end{equation*}   
\item If the decomposition
\begin{equation}\label{decomp}
X=\ker A \oplus \overline{R(A)}
\end{equation} 
holds, for each bounded weak solution $y(\cdot)$ of equation \eqref{1},
\begin{equation*}
w\mbox{-}\!\!\lim_{t\to\infty}\dfrac{1}{t}\int_0^ty(s)\,ds=Py(0),
\end{equation*} 
where $P$ is the projection onto $\ker A$ along $\overline{R(A)}$.
\end{enumerate}
\end{thm}

\begin{proof}\
\begin{enumerate}[label=(\roman*)]
\item Observe that the existence of the inverse operator $A^{-1}$ is equivalent to
\begin{equation}\label{ker}
\ker A=\{0\}
\end{equation}
and let $y(\cdot)$ be an arbitrary bounded weak solution of equation \eqref{1}. 

By the boundedness of $y(\cdot)$ and \eqref{int}, to an arbitrary sequence $\left\{t_n\right\}_{n=1}^\infty\subset (0,\infty)$ with $t_n\to \infty$, $n\to\infty$,
there corresponds a bounded sequence of elements
\begin{equation}\label{seq}
\left\{\dfrac{1}{t_n}\int_0^{t_n}y(s)\,ds\right\}_{n=1}^\infty\subset D(A).
\end{equation}

For any subsequence $\left\{t_{n(k)}\right\}_{k=1}^\infty$
of $\left\{t_n\right\}_{n=1}^\infty$, by the {\it reflexivity} $X$, the bounded subsequence
\begin{equation*}
\left\{\dfrac{1}{t_{n(k)}}\int_0^{t_{n(k)}}y(s)\,ds\right\}_{k=1}^\infty
\end{equation*}
of sequence \eqref{seq} contains a subsequence
\begin{equation*}
\left\{\dfrac{1}{t_{n(k(j))}}\int_0^{t_{n(k(j))}}y(s)\,ds\right\}_{j=1}^\infty
\end{equation*} 
such that
\begin{equation*}
w\mbox{-}\!\!\lim_{j\to\infty}\dfrac{1}{t_{n(k(j))}}\int_0^{t_{n(k(j))}}y(s)\,ds=y_\infty
\end{equation*}
for some $y_\infty\in X$.

By Proposition \ref{part},
\begin{equation*}
y_\infty\in \ker A,
\end{equation*}
and hence, by \eqref{ker}, $y_\infty=0$.

Thus, an arbitrary subsequence of sequence
\eqref{seq} contains a subsequence weakly convergent to $0$, which implies that
\begin{equation*}\label{weak2}
w\mbox{-}\!\!\lim_{n\to\infty}\dfrac{1}{t_n}\int_0^{t_n}y(s)\,ds=0.
\end{equation*}

Since, $\left\{t_n\right\}_{n=1}^\infty\subset (0,\infty)$ with $t_n\to \infty$, $n\to\infty$, is arbitrary, we conclude that
\begin{equation*}
w\mbox{-}\!\!\lim_{t\to\infty}\dfrac{1}{t}\int_0^{t}y(s)\,ds=0.
\end{equation*}
\item Decomposition \eqref{decomp} implies that the closed linear operator $A+P$ has an inverse $(A+P)^{-1}$. Indeed, suppose that, for some $f\in D(A)$, 
\[(A+P)f=0.\]
Then $Af=-Pf$, which, since the subspaces $\ker A$ and $\overline{R(A)}$ are \textit{disjoint}, implies 
\[Af=Pf=0.\]
Whence, we infer that $f\in \ker A$, and therefore, $f=Pf=0$, which proves the existence of $(A+P)^{-1}$.

Let $y(\cdot)$ be an arbitrary bounded weak solution of equation \eqref{1}.

The vector function
\[u(t):=y(t)-Py(0),\ t\ge 0,\]
obtained by combining the bounded weak solution
$y(\cdot)$ of equation \eqref{1} with its \textit{equilibrium solution} $Py(0)$, as readily follows from the linear homogeneity of equation \eqref{1}, reflected in Definition \ref{def1}, is also a bounded weak solution of \eqref{1}, i.e., for all $t\ge 0$,
\begin{equation}\label{uws2}
\int_0^t u(s)\,ds\in D(A)\ \text{and} \ u(t)=u(0)+A\int_0^t u(s)\,ds.
\end{equation}

Since, by \eqref{decomp}, $u(0)=y(0)-Py(0)\in \overline{R(A)}$, in view of \eqref{uws2}, we infer that
\[u(t)\in \overline{R(A)},\ t\ge 0,\] 
which, by the \textit{closedness} of $\overline{R(A)}$, implies that
\begin{equation*}
\int_0^t u(s)\,ds\in D(A)\cap \overline{R(A)},\ t\ge 0.
\end{equation*}
Hence,
\begin{equation}\label{null2}
P\int_0^t u(s)\,ds=0,\ t\ge 0.
\end{equation}

Now, let us show that $u(\cdot)$ is also a bounded weak solution of the evolution equation
\begin{equation}\label{3}
y'(t)=(A+P)y(t),\ t\ge 0.
\end{equation}

Indeed, by \eqref{uws2},
\begin{equation*}
\int_0^t u(s)\,ds\in D(A)=D(A+P),\ t\ge 0,
\end{equation*}
and, considering \eqref{null2} 
\begin{equation*}
u(t)=u(0)+A\int_0^t u(s)\,ds=u(0)+(A+P)\int_0^t u(s)\,ds,\ t\ge 0.
\end{equation*}

The operator $A+P$ having an inverse $(A+P)^{-1}$, according to proved part (i),
\begin{equation*}
w\mbox{-}\!\!\lim_{t\to\infty}\dfrac{1}{t}\int_0^t u(s)\,ds=0.
\end{equation*}
Whence,
\begin{equation*}
w\mbox{-}\!\!\lim_{t\to\infty}\dfrac{1}{t}\int_0^t y(s)\,ds=Py(0).
\end{equation*}
\end{enumerate}
\end{proof}

\subsection{Concluding Remarks}
\begin{enumerate}[label=\arabic*.]
\item Unlike in part (i) of Theorem \ref{Thm2}, in part (i) of Theorem \ref{Thm3}, the inverse operator $A^{-1}$ need not be bounded. 
\item In view of the \textit{reflexivity} of $l_2$, the examples given in the concluding remarks to the prior section also demonstrate that, in all, no part of Theorem \ref{Thm3} is more general than the other.
\item When $(X,\|\cdot\|)$ is a finite-dimensional space, part (ii) of Theorem \ref{Thm3} is, obviously, more general than part (i) and, the space being complex, the same argument as in the corresponding concluding remark of the prior section applies to explain the consistency of the former with Theorem \ref{Thm1}.
\end{enumerate}

\section{Scalar Type Spectral Operator}

\subsection{Preliminaries}\label{prelim1}

Henceforth, $A$ is a {\it scalar type spectral operator} in a complex Banach space $(X,\|\cdot\|)$ and $E_A(\cdot)$ be its {\it spectral measure} (the {\it resolution of the identity}), the operator's \textit{spectrum} $\sigma(A)$ being the {\it support} for the latter \cite{Survey58,Dun-SchIII}.

For the reader's convenience, we briefly outline here certain essential preliminaries regarding such operators.

Observe that, in a complex finite-dimensional space, 
the scalar type spectral operators are those linear operators on the space, for which there is an \textit{eigenbasis} (see, e.g., \cite{Survey58,Dun-SchIII}) and, in a complex Hilbert space, the scalar type spectral operators are precisely those that are similar to the {\it normal} ones \cite{Wermer}.

Associated with a scalar type spectral operator in a complex Banach space is the {\it Borel operational calculus} analogous to that for a \textit{normal operator} in a complex Hilbert space \cite{Survey58,Dun-SchII,Dun-SchIII,Plesner}, which assigns to any Borel measurable function $F:\sigma(A)\to \C$ a scalar type spectral operator
\begin{equation*}
F(A):=\int\limits_{\sigma(A)} F(\lambda)\,dE_A(\lambda)
\end{equation*}
defined as follows:
\[
F(A)f:=\lim_{n\to\infty}F_n(A)f,\ f\in D(F(A)),\
D(F(A)):=\left\{f\in X\middle| \lim_{n\to\infty}F_n(A)f\ \text{exists}\right\},
\]
where
\begin{equation*}
F_n(\cdot):=F(\cdot)\chi_{\{\lambda\in\sigma(A)\,|\,|F(\lambda)|\le n\}}(\cdot),
\ n\in\N,
\end{equation*}
($\chi_\delta(\cdot)$ is the {\it characteristic function} of a set $\delta\subseteq \C$, $\N:=\left\{1,2,3,\dots\right\}$ is the set of \textit{natural numbers}) and
\begin{equation*}
F_n(A):=\int\limits_{\sigma(A)} F_n(\lambda)\,dE_A(\lambda),\ n\in\N,
\end{equation*}
are {\it bounded} scalar type spectral operators on $X$ defined in the same manner as for a {\it normal operator} (see, e.g., \cite{Dun-SchII,Plesner}).

In particular,
\begin{equation*}
A^n=\int\limits_{\sigma(A)} \lambda^n\,dE_A(\lambda),\ n\in\Z_+,
\end{equation*}
($\Z_+:=\left\{0,1,2,\dots\right\}$ is the set of \textit{nonnegative integers}, $A^0:=I$) and
\begin{equation}\label{exp2}
e^{zA}:=\int\limits_{\sigma(A)} e^{z\lambda}\,dE_A(\lambda),\ z\in\C.
\end{equation}

The properties of the {\it spectral measure} and {\it operational calculus}, exhaustively delineated in \cite{Survey58,Dun-SchIII}, underlie the entire subsequent discourse. Here, we touch upon a few facts of particular importance.

Due to its {\it strong countable additivity}, the spectral measure $E_A(\cdot)$ is {\it bounded} \cite{Dun-SchI,Dun-SchIII}, i.e., there is such an $M\ge 1$ that, for any Borel set $\delta\subseteq \C$,
\begin{equation}\label{bounded}
\|E_A(\delta)\|\le M.
\end{equation}
Observe that the notation $\|\cdot\|$ is recycled here to designate the norm in the space $L(X)$ of all bounded linear operators on $X$. We adhere to this rather conventional economy of symbols in what follows adopting the same notation for the norm in the dual space $X^*$ as well (cf. \cite{Engel-Nagel,Markin2002(2)}). 

For any $f\in X$ and $g^*\in X^*$, the \textit{total variation} $v(f,g^*,\cdot)$ of the complex-valued Borel measure $\langle E_A(\cdot)f,g^* \rangle$ is a {\it finite} positive Borel measure with
\begin{equation}\label{tv}
v(f,g^*,\C)=v(f,g^*,\sigma(A))\le 4M\|f\|\|g^*\|
\end{equation}
(see, e.g., \cite{Markin2004(1),Markin2004(2)}).

Also (Ibid.), for a Borel measurable function $F:\C\to \C$, $f\in D(F(A))$, $g^*\in X^*$, and a Borel set $\delta\subseteq \C$,
\begin{equation}\label{cond(ii)}
\int\limits_\delta|F(\lambda)|\,dv(f,g^*,\lambda)
\le 4M\|E_A(\delta)F(A)f\|\|g^*\|.
\end{equation}
In particular, for $\delta=\sigma(A)$,
\begin{equation}\label{cond(i)}
\int\limits_{\sigma(A)}|F(\lambda)|\,d v(f,g^*,\lambda)\le 4M\|F(A)f\|\|g^*\|.
\end{equation}

Observe that the constant $M\ge 1$ in \eqref{tv}--\eqref{cond(i)} is from 
\eqref{bounded}.

Our principal result heavily relies on the following three key statements proved in \cite{Markin2002(1),Markin2002(2),Markin2006}. 

\begin{thm}[{\cite[Theorem $4.2$]{Markin2002(1)}}]\label{GWS}\ \\
Let $A$ be a scalar type spectral operator in a complex Banach space $(X,\|\cdot\|)$. A vector function $y:[0,\infty) \to X$ is a weak solution 
of equation \eqref{1} iff there is an $\displaystyle f \in \bigcap_{t\ge 0}D(e^{tA})$ such that
\begin{equation}\label{expf}
y(t)=e^{tA}f,\ t\ge 0,
\end{equation}
the operator exponentials understood in the sense of the Borel operational calculus (see \eqref{exp2}).
\end{thm}

\begin{prop}[{\cite[Proposition $3.1$]{Markin2002(2)}}]\label{SG}\ \\
A scalar type spectral operator $A$ in a complex Banach space $(X,\|\cdot\|)$ generates a $C_0$-semigroup of bounded linear operators iff there is an $\omega\in\R$
such that
\[
\sigma(A)\subseteq \left\{\lambda\in\C\,\middle|\, \Rep\lambda\le \omega\right\},
\] 
in which case the semigroup that of
the operator exponentials $\left\{e^{tA}\right\}_{t\ge 0}$.
\end{prop}

\begin{thm}[{\cite[Theorem]{Markin2006}}]\label{DSD}\ \\
For a scalar type spectral operator $A$ in a complex Banach space $(X,\|\cdot\|)$ with spectral measure $E_A(\cdot)$, the \textit{direct sum decomposition}
\begin{equation*}
X=\ker A\oplus \overline{R(A)}
\end{equation*}
holds with 
\[
\ker A=E_A(\{0\})X\quad \text{and}\quad \overline{R(A)}=E_A(\sigma(A)\setminus\{0\})X.
\]
\end{thm}

\subsection{Main Results}

\begin{thm}\label{Thm4} 
If $A$ is a scalar type spectral operator in a complex Banach space $(X,\|\cdot\|)$ with spectral measure $E_A(\cdot)$, for each bounded weak solution $y(\cdot)$ of equation \eqref{1},
\begin{equation*}
s\mbox{-}\!\!\lim_{t\to\infty}\dfrac{1}{t}\int_0^t y(s)\,ds=Py(0),
\end{equation*}
where $P=E_A(\{0\})$ is the projection onto $\ker A$ along $\overline{R(A)}$.

The spectrum of $A$ containing no pure imaginary values,
\begin{equation*}
s\mbox{-}\!\!\lim_{t\to\infty}y(t)=Py(0).
\end{equation*}
\end{thm}

\begin{proof}\quad
Let $y(\cdot)$ be an arbitrary \textit{bounded} weak solution of equation \eqref{1}. Then, by Theorem \ref{GWS},
\begin{equation*}
y(t)=e^{tA}f,\ t\ge 0,
\end{equation*}
with some $y(0)=f\in\bigcap\limits_{t\ge 0}D(e^{tA})$.

Let us show that the boundedness of $y(\cdot)$ necessarily implies that
\begin{equation}\label{zero}
E_A(\{\lambda\in\sigma(A)|\Rep\lambda>0\})f=0.
\end{equation}

Indeed, by the strong continuity of the {\it spectral measure}, the opposite indicates that there is an $n\in \N$ such that 
\begin{equation*}
E_A(\{\lambda\in\sigma(A)|\Rep\lambda\ge 1/n\})f\neq 0,
\end{equation*}
and hence, as follows from the {\it Hahn-Banach Theorem}, there is a $g^*\in X^*\setminus\{0\}$ such that
\begin{equation}\label{nonzero}
\langle E_A(\{\lambda\in\sigma(A)|\Rep\lambda\ge 1/n\})f,g^*\rangle \neq 0.
\end{equation} 

For any $t\ge 0$,
\begin{multline*}
\|y(t)\|=\|e^{tA}f\|
\hfill
\text{by \eqref{cond(i)};}
\\
\shoveleft{
\ge \left[4M\|g^*\|\right]^{-1}\int\limits_{\sigma(A)}e^{t\Rep\lambda}\,dv(f,g^*,\lambda)
}\\
\shoveleft{
\ge \left[4M\|g^*\|\right]^{-1}
\int\limits_{\{\lambda\in\sigma(A)|\Rep\lambda\ge 1/n\}}e^{t\Rep\lambda}\,dv(f,g^*,\lambda)
}\\
\shoveleft{
\ge 
\left[4M\|g^*\|\right]^{-1}e^{t/n}v(f,g^*,\{\lambda\in\sigma(A)|\Rep\lambda\ge 1/n\})
}\\
\shoveleft{
\ge 
\left[4M\|g^*\|\right]^{-1}
e^{t/n}|\langle E_A(\{\lambda\in\sigma(A)|\Rep\lambda\ge 1/n\})f,g^*\rangle|
\hfill
\text{by \eqref{nonzero};}
}\\
\ \
\to\infty,\ t\to\infty,
\hfill
\end{multline*}
($M\ge 1$ is  from \eqref{bounded}), which contradicts the boundedness of $y(\cdot)$ proving \eqref{zero}.

For the scalar type spectral operator
\begin{equation*}
A_-:=AE_A\left(\left\{\lambda \in \sigma(A)\,
\middle|\,\Rep\lambda\le 0\right\}\right),
\end{equation*}
by the properties of the operational calculus (see {\cite[Theorem XVIII.$2.11$]{Dun-SchIII}}),
\begin{equation}\label{spleft}
\sigma(A_-)\subseteq \left\{\lambda \in \sigma(A)\,
\middle|\,\Rep \lambda\le 0\right\}
\end{equation}
which implies by {\cite[Proposition $3.1$]{Markin2002(2)}} (cf. \cite{Panchapagesan1969}) that the operator $A_-$ generates the $C_0$-semigroup of its exponentials:
\begin{equation}\label{exp3}
e^{tA_-}=e^{tA}E_A\left(\left\{\lambda \in \sigma(A)\,
\middle|\,\Rep\lambda\le 0\right\}\right)
+E_A\left(\left\{\lambda \in \sigma(A)\,
\middle|\,\Rep\lambda>0\right\}\right),\ t\ge 0.
\end{equation}

Furthermore, by the properties of the operational calculus (see {\cite[Theorem XVIII.$2.11$]{Dun-SchIII}}),
inclusion \eqref{spleft} implies that the semigroup
$\left\{e^{tA_-}\right\}_{t\ge 0}$ is \textit{bounded}.

In view of \eqref{zero}, 
\begin{equation*}
f=E_A(\{\lambda\in\sigma(A)|\Rep\lambda\le 0\})f,
\end{equation*}
and hence, representation \eqref{exp3} yields
\[
y(t)=e^{tA}f=e^{tA_-}f,\ t\ge 0,
\]
which, by  
{\cite[Corollary 1]{Goldstein-Radin-Showalter}} and Theorem \ref{DSD}, implies that
\begin{equation}\label{lim}
s\mbox{-}\!\!\lim_{t \to \infty}\frac{1}{t} \int_0^t e^{tA}f\,ds
=
s\mbox{-}\!\!\lim_{t \to \infty}\frac{1}{t} \int_0^t e^{tA_-}f\,ds=E_{A_-}(\{0\})f.
\end{equation}

Since 
\[
A_-=F(A)\ \text{with}\ F(z):=z\chi_{\left\{\lambda \in \sigma(A)\,\middle|\,\Rep\lambda\le 0\right\}}(z),\
z\in \C,
\]
by {\cite[Theorem $3.3$]{Bade1954}} (see also \cite{Dun-SchIII}),
\begin{multline*}
E_{A_-}(\{0\})=E_{A}\left(F^{-1}(\{0\})\right),
=E_{A}\left(\{0\}\cup \left\{\lambda \in \sigma(A)\,
\middle|\,\Rep\lambda> 0\right\}\right)
\\
\ \
=E_{A}\left(\{0\}\right)+E_{A}\left(\left\{\lambda \in \sigma(A)\,
\middle|\,\Rep\lambda> 0\right\}\right).
\hfill
\end{multline*}

Considering the latter, by \eqref{zero}, \eqref{lim} implies
\begin{equation*}
s\mbox{-}\!\!\lim_{t \to \infty}\frac{1}{t} \int_0^t e^{tA}f\,ds
=E_{A}(\{0\})f.
\end{equation*}

Let us show that
\begin{equation}\label{neg}
s\mbox{-}\!\!\lim_{t\to\infty}
e^{tA}E_A(\{\lambda\in\sigma(A)|\Rep\lambda<0\})f=0.
\end{equation}

Let $\eps>0$ be arbitrary and, by the strong continuity of the spectral measure, fix an $n\in\N$ such that
\begin{equation}\label{eest1}
4M\left\|E_A\left(\{\lambda\in\sigma(A)|-1/n<\Rep\lambda<0\}\right)f\right\|<\eps/2
\end{equation}
and let $T>0$ be such that, for each $t\ge T$,
\begin{equation}\label{eest2}
4M\|f\|e^{-t/n}<\eps/2,
\end{equation}
the constant $M\ge 1$ in \eqref{eest1} and \eqref{eest2}
being from \eqref{bounded}.

Then, for each $t\ge T$,
\begin{multline*}
\left\|e^{tA}E_A(\{\lambda\in\sigma(A)|\Rep\lambda<0\})f\right\|
\\
\shoveleft{
\hfill
\text{by the properties of the \textit{operational calculus};}
}\\
\shoveleft{
=
\left\|\int\limits_{\{\lambda\in\sigma(A)|\Rep\lambda<0\}}
e^{t\lambda}\,dE_A(\lambda)f\right\|
}\\
\hfill
\text{as follows from the \textit{Hahn-Banach Theorem};}
\\
\shoveleft{
=\sup_{\{g^*\in X^*|\|g^*\|=1\}}
\left|\left\langle
\int\limits_{\{\lambda\in\sigma(A)|\Rep\lambda<0\}}
e^{t\lambda}\,d E_A(\lambda)f,g^*\right\rangle
\right|
}\\
\hfill
\text{by the properties of the \textit{operational calculus};}
\\
\shoveleft{
= \sup_{\{g^*\in X^*|\|g^*\|=1\}}
\left|\int\limits_{\{\lambda\in\sigma(A)|\Rep\lambda<0\}}
e^{t\lambda}\,d\langle E_A(\lambda)f,g^*\rangle\right|
}\\
\shoveleft{
\le \sup_{\{g^*\in X^*|\|g^*\|=1\}}\int\limits_{\{\lambda\in\sigma(A)|\Rep\lambda<0\}}
\left|e^{t\lambda}\right|\,dv(f,g^*,\lambda) 
}\\
\shoveleft{
= \sup_{\{g^*\in X^*|\|g^*\|=1\}}
\int\limits_{\{\lambda\in\sigma(A)|\Rep\lambda<0\}}
e^{t\Rep\lambda}\,dv(f,g^*,\lambda)
}\\
\shoveleft{
= \sup_{\{g^*\in X^*|\|g^*\|=1\}}\biggl[
\int\limits_{\{\lambda\in\sigma(A)|\Rep\lambda\le -1/n\}}
e^{t\Rep\lambda}\,dv(f,g^*,\lambda)
}\\
\shoveleft{
+
\int\limits_{\{\lambda\in\sigma(A)|-1/n<\Rep\lambda<0\}}
e^{t\Rep\lambda}\,dv(f,g^*,\lambda)
\biggr]
}\\
\shoveleft{
\le \sup_{\{g^*\in X^*|\|g^*\|=1\}}\biggl[
\int\limits_{\{\lambda\in\sigma(A)|\Rep\lambda\le -1/n\}}
e^{-t/n}\,dv(f,g^*,\lambda)
}\\
\shoveleft{
+
\int\limits_{\{\lambda\in\sigma(A)|-1/n<\Rep\lambda<0\}}
1\,dv(f,g^*,\lambda)
\biggr]
\hfill
\text{by \eqref{tv} and \eqref{cond(ii)} with $F(\lambda)\equiv 1$};
}\\
\shoveleft{
\le \sup_{\{g^*\in X^*|\|g^*\|=1\}}\left[
e^{-t/n}4M\|f\|\|g^*\|
+4M\left\|E_A(\{\lambda\in\sigma(A)|-1/n<\Rep\lambda<0\})f\right\|\|g^*\|
\right]
}\\
\shoveleft{
\le 4M\|f\|e^{-t/n}
+4M\left\|E_A(\{\lambda\in\sigma(A)|-1/n<\Rep\lambda<0\})f\right\|
}\\
\hfill
\text{by the \eqref{eest1} and \eqref{eest2};}
\\
\ \
<\eps.
\hfill
\end{multline*}

Whence, \eqref{neg} follows.

If the spectrum of $A$ contains no pure imaginary values, by \eqref{zero} and \eqref{neg},
\begin{multline*}
s\mbox{-}\!\!\lim_{t\to\infty}y(t)
=s\mbox{-}\!\!\lim_{t\to\infty}e^{tA}
=s\mbox{-}\!\!\lim_{t\to\infty}
\left[
e^{tA}E_A(\{\lambda\in\sigma(A)|\Rep\lambda<0\})
+
E_A(\{0\})
\right]f
\\
\ \
=
s\mbox{-}\!\!\lim_{t\to\infty}e^{tA}E_A(\{\lambda\in\sigma(A)|\Rep\lambda<0\})f
+
E_A(\{0\})f=E_A(\{0\})f.
\hfill
\end{multline*}

Thus, the proof is complete.
\end{proof}

As an important particular case, we obtain the following 

\begin{thm}
If $A$ is a normal operator in a complex Hilbert space $X$, for each bounded weak solution $y(\cdot)$ of equation \eqref{1},
\begin{equation*}
s\mbox{-}\!\!\lim_{t\to\infty}\dfrac{1}{t}\int_0^t y(s)\,ds=Py(0),
\end{equation*}
where $P$ is the orthogonal projection onto $\ker A$.

The spectrum of $A$ containing no pure imaginary values,
\begin{equation*}
s\mbox{-}\!\!\lim_{t\to\infty}y(t)=Py(0).
\end{equation*}
\end{thm}

In view of the fact that the spectrum of a self-adjoint operator is located on the real axis (see, e.g., \cite{Dun-SchII,Plesner}), we instantly arrive at

\begin{cor}
If $A$ is a self-adjoint operator in a complex Hilbert space, for each bounded weak solution $y(\cdot)$ of equation \eqref{1},
\begin{equation*}
s\mbox{-}\!\!\lim_{t\to\infty}y(t)=Py(0),
\end{equation*}
where $P$ is the orthogonal projection onto $\ker A$.
\end{cor}

The latter generalizes the complex version of {\cite[Proposition 24]{Haraux}}, which states that, $A$ being a nonpositive self-adjoint operator in a Hilbert space $X$, for any $f\in X$,
\begin{equation*}
s\mbox{-}\!\!\lim_{t\to\infty}e^{tA}f=Py(0),
\end{equation*}
where $P$ is the orthogonal projection onto $\ker A$.

\subsection{Concluding Remarks}
\begin{enumerate}[label=\arabic*.]
\item As follows from
Theorem \ref{DSD}
and {\cite[Theorem 3.2]{Markin2017(1)}},
a scalar type spectral operator $A$ in a complex Banach space $(X,\|\cdot\|)$
is reducibly invertible \textit{iff} $0$ is either its \textit{regular point} or an \textit{isolated point of spectrum} (see also {\cite[Corollary 3.2]{Markin2017(1)}}). In this case, Theorem \ref{Thm4} is consistent with part (ii) of Theorem \ref{Thm2}.
\item When $(X,\|\cdot\|)$ is a complex reflexive space, Theorem \ref{Thm4} is also consistent with part (ii) of Theorem \ref{Thm3}. 
\item When $(X,\|\cdot\|)$ is a complex finite-dimensional space, 
an operator $A\in L(X)$
is scalar type spectral \textit{iff} $X$ has an \textit{eigenbasis} for $A$,
i.e., the Jordan canonical matrix representation of $A$
is a \textit{diagonal matrix} 
(see, e.g., \cite{Dunford1954,Survey58,Dun-SchII}).
In this case, $0$ being either a \textit{regular point} of $A$ or an \textit{eigenvalue}, whose {\it index} is equal to $1$, i.e., all Jordan blocks corresponding to $0$ are of size 1, Theorem \ref{Thm4} is consistent with Theorem \ref{Thm1} as well.
\end{enumerate}

\section{Dedication}

With deep sadness and utmost appreciation, 
I dedicate this work
to the loving memory of my recently departed teacher,
Dr. Miroslav L. Gorbachuk, whose
life and work have so profoundly inspired and influenced many mathematicians, whom I am blessed and honored to be one of.

 

\begin{thebibliography}{99}
\bibitem{Bade1954}
{W.G. Bade},	
\textit{Unbounded spectral operators},	
{Pacific J. Math.}	
\textbf{4}	
{(1954)},	
{373--392}.	
\bibitem{Ball}
{J.M. Ball},	
\textit{Strongly continuous semigroups, weak solutions, and the variation of constants formula},	
{Proc. Amer. Math. Soc.}	
\textbf{63}	
{(1977)},	
{no.~2},	
{101--107}.	
\bibitem{Butyrin1994}
{A.A. Butyrin},	
\textit{On the behavior of solutions of operator-differential equations at infinity},	
{Ukrainian Math. J.}	
\textbf{46}	
{(1994)},	
{no.~7},	
{885--890}.	
\bibitem{Daletsky-Krein}
{Yu.L. Daletski\u{\i} and M.G. Krein},	
\textit{The Stability of Solutions of Differential Equations in a Banach Space},	
{Nauka},	
{Moscow},		
{1970}	
{(Russian)}.
\bibitem{Dunford1954}
{N. Dunford},	
\textit{Spectral operators},	
{Pacific J. Math.}	
\textbf{4}	
{(1954)},	
{321--354}.	
\bibitem{Survey58}
{\bysame},	
\textit{A survey of the theory of spectral operators}, 
{Bull. Amer. Math. Soc.}	
\textbf{64}	
{(1958)},	
{217--274}.	
\bibitem{Dun-SchI}
{N. Dunford and J.T. Schwartz with the assistance of W.G. Bade and R.G. Bartle},	
\textit{Linear Operators. Part I: General Theory},	
{Interscience Publishers},	
{New York},		
{1958}.		
\bibitem{Dun-SchII}
{\bysame},	
\textit{Linear Operators. Part II: Spectral Theory. Self Adjoint Operators in Hilbert Space}, 
{Interscience Publishers},	
{New York},		
{1963}.		
\bibitem{Dun-SchIII}
{\bysame},	
\textit{Linear Operators. Part III: Spectral Operators}, 
{Interscience Publishers},	
{New York},		
{1971}.		
\bibitem{Engel-Nagel}
{K.-J. Engel and R. Nagel},	
\textit{One-Parameter Semigroups
for Linear Evolution Equations}, 
{Graduate Texts in Mathematics, vol. 194},	
{Springer-Verlag},	
{New York},		
{2000}.		
\bibitem{Engel-Nagel2006}
{\bysame},	
\textit{A Short Course on Operator Semigroups}, 
{Universitext},	
{Springer},	
{New York},		
{2006}.		
\bibitem{Glazman-Lyubich}
{I.M. Glazman and Yu.I. Lyubich},	
\textit{Finite-Dimensional Linear Analysis},	
{Nauka},	
{Moscow},		
{1969}	
{(Russian)}.
\bibitem{Goldberg}
{S. Goldberg},	
\textit{Unbounded Linear Operators: Theory and Applications},	
{McGraw-Hill},	
{New York},		
{1966}.		
\bibitem{Goldstein-Radin-Showalter}
{J. Goldstein, C. Radin, and R.E. Showalter},	
\textit{Convergence rates of ergodic limits for semigroups and cosine functions},	
{Semigroup Forum}	
\textbf{16}	
{(1978)},	
{no.~1},	
{89--95}.	
\bibitem{Haraux}
{A. Haraux},	
\textit{Nonlinear Evolution Equations ---
Global Behavior of Solutions}, 
{Lecture Notes in Mathematics, vol. 841},	
{Springer-Verlag},	
{Berlin-New York},  
{1981}.		
\bibitem{Hille-Phillips}
{E. Hille and R.S. Phillips},	
\textit{Functional Analysis and Semi-groups},	
{American Mathematical Society Colloquium Publications, vol.~31},	
{Amer. Math. Soc.},	
{Providence, RI},		
{1957}.		
\bibitem{Horn-Johnson}
{R.A. Horn and C.R. Johnson},	
\textit{Matrix Analysis}, 
{Cambridge University Press},	
{New York},		
{1986}.		
\bibitem{Kato}
{T. Kato},	
\textit{Perturbation Theory for Linear Operators}, 
{Reprint of the 1980 Edition},	
{Classics in Mathematics},	
{Springer-Verlag},	
{Berlin-Heidelberg},		
{1995}.		
\bibitem{Komatsu1969}
{H. Komatsu},	
\textit{Fractional powers of operators. III. Negative powers},	
{J. Math. Soc. Japan}	
\textbf{21}	
{(1969)},	
{no.~2},	
{205--220}.	
\bibitem{Korolyuk-Turbin}
{V.S. Korolyuk and A.F. Turbin},	
\textit{The Mathematical Foundations of the Phase Consolidation of Complex Systems},	
{Naukova Dumka},	
{Kiev},		
{1978}	
{(Russian)}.
\bibitem{Lin-Shaw}
{C.S.C. Lin and S.-Y. Shaw},	
\textit{Ergodic theorems of semigroups and application},	
{Bull. Inst. Math. Acad. Sinica}	
\textbf{6}	
{(1978)},	
{no.~1},	
{181--188}.	
\bibitem{Lyapunov1892}
{A.M. Lyapunov},	
\textit{Stability of Motion},	
{Ph.D. Thesis},	
{Kharkov},	
{1892},		
{English Translation},	
{Academic Press},	
{New York-London},		
{1966}.	
\bibitem{Markin1991}
{M.V. Markin},	
\textit{On the behavior at infinity of bounded solutions of differential equations in a Banach space},	
{Boundary Value Problems for Operator-Differential Equations},	
{56--63},	
{Akad. Nauk Ukrainy, Inst. Mat.},	
{Kiev},		
{1991}		
{(Russian)}.
\bibitem{Markin1994Preprint} 
{\bysame},	
\textit{Ergodicity of weak solutions of a first-order operator-differential equation}, 
{Akad. Nauk Ukrainy, Inst. Mat.},	
{Preprint 1994},	
{no.~10},	
{44 pp.}	
{(Ukrainian)}.
\bibitem{Markin1994Disser}
{\bysame},	
\textit{The Smoothness and Ergodicity of Weak Solutions of a First-Order Operator-Differential Equation},	
{Ph.D. Thesis},	
{Inst. Math. Nat. Acad. Sci. Ukraine},	
{Kiev},	
{1994}		
{(Ukrainian)}.
\bibitem{Markin2002(1)}
{\bysame},	
\textit{On an abstract evolution equation with a spectral operator of scalar type},	
{Int. J. Math. Math. Sci.}	
\textbf{32}	
{(2002)},	
{no.~9},	
{555--563}.	
\bibitem{Markin2002(2)}
{\bysame},	
\textit{A note on the spectral operators of scalar type and semigroups of bounded linear operators},	
{Ibid.}	
\textbf{32}	
{(2002)},	
{no.~10},	
{635--640}.	
\bibitem{Markin2004(1)}
{\bysame},	
\textit{On scalar type spectral operators, infinite differentiable and Gevrey ultradifferentiable $C_0$-semigroups},	
{Ibid.}	
\textbf{2004}	
{(2004)},	
{no.~45},	
{2401--2422}.	
\bibitem{Markin2004(2)}
{\bysame},	
\textit{On the Carleman classes of vectors of a scalar type spectral operator},	
{Ibid.}	
\textbf{2004}	
{(2004)},	
{no.~60},	
{3219--3235}.	
\bibitem{Markin2006}
{\bysame},	
\textit{A note on one decomposition of Banach spaces},	
{Methods Funct. Anal. Topol.}	
\textbf{12}	
{(2006)},	
{no.~3},	
{254--257}.	
\bibitem{Markin2017(1)}
{\bysame},	
\textit{On certain spectral features inherent to scalar type spectral operators},	
{Ibid.}	
\textbf{23}	
{(2017)},	
{no.~1},	
{60--65}.	
\bibitem{Panchapagesan1969}
{T.V. Panchapagesan},	
\textit{Semi-groups of scalar type operators in Banach
spaces},	
{Pacific J. Math.}	
\textbf{30}	
{(1969)},	
{no.~2},	
{489--517}.	
\bibitem{Plesner}
{A.I. Plesner},	
\textit{Spectral Theory of Linear Operators},	
{Nauka},	
{Moscow},		
{1965}		
{(Russian)}.
\bibitem{Shaw}
{S.-Y. Shaw},	
\textit{Ergodic projections of continuous and discrete semigroups},	
{Proc. Amer. Math. Soc.}	
\textbf{78}	
{(1980)},	
{no.~1},	
{69--76}.	
\bibitem{Trenogin}
{V.A. Trenogin},	
\textit{Functional Analysis},	
{Nauka},	
{Moscow},		
{1980}		
{(Russian)}.
\bibitem{Wermer}
{J. Wermer},	
\textit{Commuting spectral measures on Hilbert space},	
{Pacific J. Math.}	
\textbf{4}	
{(1954)},	
{355--361}.	
\end{thebibliography}
\end{document}